\documentclass[11pt,a4paper]{article}

\usepackage{inputenc}
\usepackage{amsmath}
\usepackage{bm}
\usepackage{bbold}
\usepackage{amsthm}
\usepackage{enumerate}

\usepackage{hyperref}

\title{A Maximization Problem in Tropical Mathematics: A Complete Solution and Application Examples}

\author{Nikolai Krivulin\thanks{Faculty of Mathematics and Mechanics, St.~Petersburg State University, 28 Universitetsky Ave., St.~Petersburg, 198504, Russia, 
nkk@math.spbu.ru.}
}

\date{}

\newtheorem{theorem}{Theorem}
\newtheorem{lemma}[theorem]{Lemma}
\newtheorem{corollary}[theorem]{Corollary}

\begin{document}

\maketitle

\begin{abstract}
A multidimensional optimization problem is formulated in the tropical mathematics setting as to maximize a nonlinear objective function, which is defined through a multiplicative conjugate transposition operator on vectors in a finite-dimensional semimodule over a general idempotent semifield. The study is motivated by problems drawn from project scheduling, where the deviation between initiation or completion times of activities in a project is to be maximized subject to various precedence constraints among the activities. To solve the unconstrained problem, we first establish an upper bound for the objective function, and then solve a system of vector equations to find all vectors that yield the bound. As a corollary, an extension of the solution to handle constrained problems is discussed. The results obtained are applied to give complete direct solutions to the motivating problems from project scheduling. Numerical examples of the development of optimal schedules are also presented.
\\

\textbf{Key-Words:} tropical mathematics, idempotent semifield, optimization problem, nonlinear objective function, project scheduling.
\\

\textbf{MSC (2010):} 65K10, 15A80, 90C48, 90B35
\end{abstract}

\section{Introduction}

Optimization problems that are formulated and solved in the framework of tropical mathematics offer an evolving research domain in applied mathematics with an expanding application scope. Tropical (idempotent) mathematics deals with semirings with idempotent addition and dates back to pioneering works \cite{Pandit1961Anew,CuninghameGreen1962Describing,Hoffman1963Onabstract,Giffler1963Scheduling,Vorobjev1963Theextremal,Romanovskii1964Asymptotic}, which were inspired by real-world problems in operations research, including optimization problems. 

The tropical optimization problems under consideration are set up in the tropical mathematics setting to minimize or maximize linear and nonlinear functions defined on finite-dimensional semimodules over idempotent semifields, subject to linear inequality and equality constraints. The linear objective functions turn the problems into formal idempotent analogues of ordinary linear programming problems. The nonlinear objective functions are assumed to be defined through a multiplicative conjugate transposition operator. 

There is a range of solution approaches offered to handle particular problems in a set of works, which include \cite{Hoffman1963Onabstract,Cuninghamegreen1976Projections,Superville1978Various,Zimmermann1984Some,Butkovic2009Introduction,Gaubert2012Tropical}. Among them are iterative algorithms that produce a solution if any, or indicate that no solution exists otherwise \cite{Zimmermann1984Some,Zimmermann2006Interval,Butkovic2009Introduction,Gaubert2012Tropical}, and exact methods that provide direct solutions in a closed form \cite{Hoffman1963Onabstract,Cuninghamegreen1976Projections,Superville1978Various,Zimmermann2003Disjunctive,Zimmermann2006Interval}. Many problems are represented and worked out in terms of particular idempotent semifields as those in \cite{Superville1978Various,Zimmermann1984Some,Butkovic2009Introduction,Gaubert2012Tropical}, whereas some other problems are examined in a general setting, which covers the above semifields as special cases \cite{Hoffman1963Onabstract,Cuninghamegreen1976Projections,Zimmermann2006Interval}. Existing methods, however, mainly give a particular solution, rather than all solutions to the problem under study.

As the problems can appear in a variety of applied contexts, a large body of motivation and application examples is drawn from optimal scheduling \cite{Cuninghamegreen1976Projections,Zimmermann1984Some,Zimmermann2006Interval,Butkovic2009Onsome,Tam2010Optimizing}. Specifically, the examples include scheduling problems, where the objective function takes the form of the span (range) seminorm.

The span seminorm is defined, in the ordinary setting, as the maximum deviation between components of a vector. It finds application as an optimality criterion in diverse areas from the analysis of Markov decision processes \cite{Bather1973Optimal,Puterman2005Markov} to the form-error measurement in precision metrology \cite{Murthy1980Minimum,Gosavi2012Form}. 

In the context of tropical mathematics, the span seminorm is introduced by \cite{Cuninghamegreen1979Minimax,CuninghameGreen2004Bases}, where it is called the range seminorm. Both problems of minimizing and maximizing the seminorm taken from machine scheduling are examined in \cite{Butkovic2009Onsome,Tam2010Optimizing} with a combined technique, which needs to use two reciprocally dual idempotent semifields.

Another more straightforward approach is implemented in \cite{Krivulin2013Explicit} to solve problems of minimizing the span seminorm, where the seminorm is represented as a nonlinear objective function defined through a conjugate transposition operator. The problem arises in project management within the framework of just-in-time scheduling of activities constrained by various precedence relations (see, e.g., \cite{Tkindt2006Multicriteria,Demeulemeester2002Project} for further details and references on project scheduling). Based on the approach, a direct exact solution to the problems is obtained in a compact vector form given in terms of a single semiring.

In this paper, we start with the same problems as above, except that the span seminorm is maximized. In the context of optimal scheduling, the problems appear when activity initiation or completion times are to be spread over the maximum possible time interval due to the lack of resource to handle all activities simultaneously. One of the problems, which is to maximize the completion time deviation of activities, is similar to that considered in \cite{Butkovic2009Onsome,Tam2010Optimizing}.

We formulate a common tropical optimization problem as to maximize a nonlinear objective function defined on vectors over a general idempotent semifield. To solve the problem in terms of the carrier semiring, we first establish an upper bound for the objective function, and then find all vectors that yield the bound. As particular cases, complete direct solutions are given to the problems of maximizing the span seminorm in project scheduling.

The rest of the paper is organized as follows. Section~\ref{S-Me} suggests motivating problems coming from project scheduling. In Section~\ref{S-Pdr}, we give an overview of preliminary definitions and results of idempotent algebra, including complete solutions to linear vector equations. The main result, which offers a complete direct solution to a general maximization problem, and its corollaries are given in Section~\ref{S-Op}. Finally, we present applications of the results obtained to solve scheduling problems together with numerical examples in Section~\ref{S-Aps}.

\section{Motivating examples}\label{S-Me}

In this section, we describe problems drawn from the project scheduling \cite{Tkindt2006Multicriteria,Demeulemeester2002Project} and intended to both motivate and illustrate the development of solutions to tropical optimization problems presented below. The scheduling problems are formulated in the general terms of activities and precedence relations, which can represent actual jobs, tasks or operations and time constraints placed on them by technical, operational, or other real-world limitations.

Suppose there is a project that involves certain activities operating under various temporal constraints. The constraints have the form of start-finish and start-start precedence relations defined for each pair of activities. The start-finish relation limits a minimum allowed time lag between initiation of one activity and completion of the other, whereas the start-start relation fixes a minimum lag between initiations of the activities. Each activity is assumed to complete at the earliest possible time within the constraints imposed.

Scheduling problems of interest are to determine, subject to the constraints, an appropriate initiation time for each activity so as to satisfy an optimality criterion in the form of maximum deviation time between either initiation or completion times of the activities.

Consider a project of $n$ activities. For each activity $i=1,\ldots,n$, denote the initiation time by $x_{i}$ and the completion time by $y_{i}$. Let $a_{ij}$ be a given time lag between initiation of activity $j=1,\ldots,n$ and completion of $i$. The start-finish constraints are represented in the ordinary notation by the equalities
$$
\max_{1\leq j\leq n}(x_{j}+a_{ij})
=
y_{i},
\qquad
i=1,\ldots,n.
$$

With the maximum deviation of completion time of activities given by
$$
\max_{1\leq i\leq n}y_{i}
-
\min_{1\leq i\leq n}y_{i}
=
\max_{1\leq i\leq n}y_{i}
+
\max_{1\leq i\leq n}(-y_{i}),
$$
we arrive at a problem of finding for each $i=1,\ldots,n$ the unknown $x_{i}$ that
\begin{equation}
\begin{aligned}
&
\text{maximize}
&&
\max_{1\leq i\leq n}y_{i}
+
\max_{1\leq i\leq n}(-y_{i}),
\\
&
\text{subject to}
&&
\max_{1\leq j\leq n}(x_{j}+a_{ij})
=
y_{i},
\quad
i=1,\ldots,n.
\end{aligned}
\label{P-maxyimaxyixjaijyi}
\end{equation}

Note that a similar problem arising in machine scheduling is examined in \cite{Butkovic2009Onsome,Tam2010Optimizing} in the context of the analysis of the image set of a max-linear mapping.

Furthermore, given $c_{ij}$ to be a time lag between initiation of activity $j$ and initiation of $i$, the start-start constraints yield the inequalities
$$
\max_{1\leq j\leq n}(x_{j}+c_{ij})
\leq
x_{i},
\qquad
i=1,\ldots,n.
$$

If, for some $i$ and $j$, there is actually no time lag defined, we put $c_{ij}=-\infty$.

The problem of maximizing the deviation between initiation times of activities takes the form
\begin{equation}
\begin{aligned}
&
\text{maximize}
&&
\max_{1\leq i\leq n}x_{i}
+
\max_{1\leq i\leq n}(-x_{i}),
\\
&
\text{subject to}
&&
\max_{1\leq j\leq n}(x_{j}+c_{ij})
\leq
x_{i},
\quad
i=1,\ldots,n.
\end{aligned}
\label{P-maxximaxxixjcijxi}
\end{equation}

Finally, when both start-finish and start-start constraints are taken into account, we get a problem to find an initiation time for each activity to
\begin{equation}
\begin{aligned}
&
\text{maximize}
&&
\max_{1\leq i\leq n}y_{i}
+
\max_{1\leq i\leq n}(-y_{i}),
\\
&
\text{subject to}
&&
\max_{1\leq j\leq n}(x_{j}+a_{ij})
=
y_{i},
\\
&
&&
\max_{1\leq j\leq n}(x_{j}+c_{ij})
\leq
x_{i},
\quad
i=1,\ldots,n.
\end{aligned}
\label{P-maxyimaxyixjcijxixjaijyi}
\end{equation}

Below the scheduling problems considered are represented in terms of tropical mathematics. We offer a complete solution to a general tropical optimization problem, and then solve the scheduling problems as particular cases.

\section{Preliminary definitions and results}\label{S-Pdr}

The purpose of this section is to give a brief overview of basic definitions and preliminary results that underlie formulation and solution of tropical optimization problems under study. In the literature, there is a range of works that provide concise introduction to as well as comprehensive coverage of the theory and methods of tropical mathematics in various forms and somewhat different formal languages, including recent publications \cite{Kolokoltsov1997Idempotent,Golan2003Semirings,Heidergott2006Maxplus,Litvinov2007Themaslov,Gondran2008Graphs,Butkovic2010Maxlinear}. 

The overview offered below is mainly based on the presentation style of notation and results in \cite{Krivulin2006Solution,Krivulin2009Methods}, which offers the possibility of getting direct complete solutions in a compact vector form. For additional details and further discussion, one can consult references listed before.

\subsection{Idempotent semifield}

Let $\mathbb{X}$ be a set that is closed with respect to addition $\oplus$ and multiplication $\otimes$, which are both associative and commutative binary operations, where multiplication is distributive over addition. The set includes zero $\mathbb{0}$ and unit $\mathbb{1}$ to be respective neutral elements for addition and multiplication. Addition is assumed to be idempotent, which implies that $x\oplus x=x$ for all $x\in\mathbb{X}$. Multiplication is invertible to provide each $x\in\mathbb{X}\setminus\{\mathbb{0}\}$ with an inverse $x^{-1}$ such that $x^{-1}\otimes x=\mathbb{1}$. Under these assumptions, the algebraic structure $\langle\mathbb{X},\mathbb{0},\mathbb{1},\oplus,\otimes\rangle$ is commonly referred to as the idempotent semifield over $\mathbb{X}$.

Idempotent addition imposes a partial order on the semifield, which establishes a relation $x\leq y$ if and only if $x\oplus y=y$. The definition implies that addition has an extremal property, which ensures the inequalities $x\leq x\oplus y$ and $y\leq x\oplus y$ for all $x,y\in\mathbb{X}$, as well as that both addition and multiplication are isotone in each argument. Finally, it is assumed that the partial order can be completed into a total order, which makes the semifield linearly ordered. 

In what follows, we routinely omit the multiplication sign for the brevity sake. The relation symbols and the $\max$ operator are thought of as defined in terms of the order induced by idempotent addition.

The semifield $\mathbb{R}_{\max,+}=\langle\mathbb{R}\cup\{-\infty\},-\infty,0,\max,+\rangle$ over the set of real numbers $\mathbb{R}$ offers an example of idempotent semifield under study, which is used to represent and solve optimal scheduling problems below.

\subsection{Matrix algebra}

Matrices and vectors with entries in $\mathbb{X}$ are routinely defined together with related operations, which are performed according to the conventional rules with the operations $\oplus$ and $\otimes$ in the role of ordinary addition and multiplication.

As usual, the set of matrices over $\mathbb{X}$ with $m$ rows and $n$ columns is denoted by $\mathbb{X}^{m\times n}$. A matrix with all entries equal to $\mathbb{0}$ is the zero matrix. A matrix is row (column) regular if it has no rows (columns) that consist entirely of $\mathbb{0}$.

In what follows, we denote matrices with bold uppercase letters. For each introduced matrix, the same bold lowercase and normal lowercase letters are reserved respectively for the columns and entries of the matrix. Specifically, a column and an entry of a matrix $\bm{A}$ are denoted by $\bm{a}_{i}$ and $a_{ij}$.

The extremal property of scalar addition extends to matrix addition in the form of entry-wise inequalities $\bm{A}\leq\bm{A}\oplus\bm{B}$ and $\bm{B}\leq\bm{A}\oplus\bm{B}$, which are valid for all $\bm{A},\bm{B}\in\mathbb{X}^{m\times n}$. Addition and multiplication of matrices, as well as multiplication of matrices by scalars, are isotone in each argument.

For any matrix $\bm{A}=(a_{ij})\in\mathbb{X}^{m\times n}$ with regular columns, there defined a multiplicative conjugate transpose as a matrix $\bm{A}^{-}=(a_{ij}^{-})\in\mathbb{X}^{n\times m}$ with entries $a_{ij}^{-}=a_{ji}^{-1}$. For two conforming matrices $\bm{A}$ and $\bm{B}$ without zero entries, the entry-wise inequality $\bm{A}\leq\bm{B}$ implies the inequality $\bm{A}^{-}\geq\bm{B}^{-}$.

Consider square matrices in $\mathbb{X}^{n\times n}$. A square matrix that has $\mathbb{1}$ on the diagonal and $\mathbb{0}$ elsewhere, is the identity matrix denoted by $\bm{I}$. The power notation with nonnegative integer exponents is used to represent repeated multiplication by the same matrix with $\bm{A}^{0}=\bm{I}$ for any $\bm{A}\in\mathbb{X}^{n\times n}$.

For any matrix $\bm{A}=(a_{ij})\in\mathbb{X}^{n\times n}$, the trace is given by
$$
\mathop\mathrm{tr}\bm{A}
=
\bigoplus_{i=1}^{n}a_{ii}.
$$

A matrix is reducible if it can be put in a block-triangular form with zero blocks above (or below) the diagonal by simultaneous permutation of rows and columns. Otherwise, the matrix is considered to be irreducible. Any matrix with regular columns (rows) has only nonzero entries, and so irreducible. 

It is not difficult to see that, for any irreducible matrix $\bm{A}\in\mathbb{X}^{n\times n}$, the matrix $\bm{I}\oplus\bm{A}\oplus\cdots\oplus\bm{A}^{n-1}$ has no zero entries.

Any matrix of one column presents a column vectors. The set of column vectors with $n$ components over $\mathbb{X}$ is denoted by $\mathbb{X}^{n}$ and forms a finite-dimensional idempotent semimodule with respect to vector addition and scalar multiplication. A vector with all components equal to $\mathbb{0}$ is the zero vector. A vector is called regular if it has no zero components.


For any regular column vector $\bm{x}=(x_{i})\in\mathbb{X}^{n}$, there defined a multiplicative conjugate transpose $\bm{x}^{-}=(x_{i}^{-})$ as a row vector with components $x_{i}^{-}=x_{i}^{-1}$. It is not difficult to verify that if a vector $\bm{x}$ is regular, then $\bm{x}\bm{x}^{-}\geq\bm{I}$. For any two regular vectors $\bm{x}$ and $\bm{y}$ of the same order, it holds that $(\bm{x}\bm{y}^{-} )^{-}=\bm{y}\bm{x}^{-}$.

To simplify some further formulae, we introduce, for any vector $\bm{x}\in\mathbb{X}^{n}$ and matrix $\bm{A}\in\mathbb{X}^{m\times n}$, idempotent analogues of the vector and matrix norms
$$
\|\bm{x}\|
=
\bigoplus_{i=1}^{n}x_{i},
\qquad
\|\bm{A}\|
=
\bigoplus_{i=1}^{m}\bigoplus_{j=1}^{n}a_{ij}.
$$

Denote by $\mathbb{1}$ a vector with all components equal to $\mathbb{1}$. Now we can write
$$
\|\bm{x}\|
=
\mathbb{1}^{T}\bm{x},
\qquad
\|\bm{A}\|
=
\mathbb{1}^{T}\bm{A}\mathbb{1}.
$$

For any vectors $\bm{x}$ and $\bm{y}$ of the same order, it holds that $\|\bm{x}\bm{y}^{T}\|=\|\bm{x}\|\|\bm{y}\|$.

\subsection{Linear equations and inequalities}

Assume $\bm{A}\in\mathbb{X}^{m\times n}$ to be a given matrix and $\bm{d}\in\mathbb{X}^{n}$ a given vector. Consider a problem to find solutions $\bm{x}\in\mathbb{X}^{n}$ to a general equation
$$
\bm{A}\bm{x}
=
\bm{d}.
$$

A complete direct solution of the problem is given in a vector form in \cite{Krivulin2009Methods,Krivulin2012Asolution}. In what follows, we need a solution to a particular case when $m=1$. Given a vector $\bm{a}\in\mathbb{X}^{n}$ and a scalar $d\in\mathbb{X}$, the problem is to solve an equation
\begin{equation}
\bm{a}^{T}\bm{x}
=
d.
\label{E-axd}
\end{equation}

Based on the solution of the general equation, a solution to \eqref{E-axd} is as follows.
\begin{lemma}
Let $\bm{a}=(a_{i})$ be a regular vector and $d>\mathbb{0}$ a scalar. Then the solutions of equation \eqref{E-axd} form a family of solutions each defined for a particular $i=1,\ldots,n$ as a set of vectors $\bm{x}=(x_{i})$ with components
\begin{align*}
x_{i}
&=
a_{i}^{-1}d,
\\
x_{j}
&\leq
a_{j}^{-1}d,
\quad
j\ne i.
\end{align*}
\label{L-axd}
\end{lemma}

Now we present solutions to another problem to be used below. Given a matrix $\bm{A}\in\mathbb{X}^{n\times n}$, consider a problem of finding regular solutions $\bm{x}\in\mathbb{X}^{n}$ that satisfy an  inequality
\begin{equation}
\bm{A}\bm{x}
\leq
\bm{x}.
\label{I-Axx}
\end{equation}

To describe a solution given in \cite{Krivulin2006Solution,Krivulin2009Methods}, we make some definitions. For each matrix $\bm{A}\in\mathbb{X}^{n\times n}$, a function is introduced that yields a scalar
$$
\mathop\mathrm{Tr}(\bm{A})
=
\mathop\mathrm{tr}\bm{A}\oplus\cdots\oplus\mathop\mathrm{tr}\bm{A}^{n}.
$$

Under the condition that $\mathop\mathrm{Tr}(\bm{A})\leq\mathbb{1}$, we further define an asterate of $\bm{A}$ to be the matrix
$$
\bm{A}^{\ast}
=
\bm{I}\oplus\bm{A}\oplus\cdots\oplus\bm{A}^{n-1}.
$$

A direct solution to inequality \eqref{I-Axx} is given by the next result.

\begin{theorem}\label{T-I-Axx}
Let $\bm{x}$ be the general regular solution of inequality \eqref{I-Axx} with an irreducible matrix $\bm{A}$. Then the following statements hold:
\begin{enumerate}
\item If $\mathop\mathrm{Tr}(\bm{A})\leq\mathbb{1}$, then $\bm{x}=\bm{A}^{\ast}\bm{u}$ for all regular vectors $\bm{u}$.
\item If $\mathop\mathrm{Tr}(\bm{A})>\mathbb{1}$, then there is no regular solution.
\end{enumerate}
\end{theorem}

\section{Optimization problem}\label{S-Op}

We are now in a position to present the main result, which offers a solution to the following tropical optimization problem. Given matrices $\bm{A}\in\mathbb{X}^{m\times n}$, $\bm{B}\in\mathbb{X}^{l\times n}$ and vectors $\bm{p}\in\mathbb{X}^{m}$ and $\bm{q}\in\mathbb{X}^{l}$, find regular solutions $\bm{x}\in\mathbb{X}^{n}$ that 
\begin{equation}
\begin{aligned}
&
\text{maximize}
&&
\bm{q}^{-}\bm{B}\bm{x}(\bm{A}\bm{x})^{-}\bm{p}.
\end{aligned}
\label{P-maxqBxAxp}
\end{equation}

Below a solution to the problem is obtained under fairly general assumptions. Furthermore, a solution is given to a particular case of the problem. An extension of the solution to handle constrained problems is also discussed.

\subsection{The main result}

The next statement offers a direct complete solution to problem \eqref{P-maxqBxAxp}. 

\begin{theorem}\label{T-maxqBxAxp}
Suppose $\bm{A}$ is a matrix with regular columns, $\bm{B}$ is a column regular matrix, $\bm{p}$ and $\bm{q}$ are regular vectors. Define a scalar
\begin{equation}
\Delta
=
\bm{q}^{-}\bm{B}\bm{A}^{-}\bm{p}.
\label{E-DeltaqBAp}
\end{equation}

Then the maximum in problem \eqref{P-maxqBxAxp} is equal to $\Delta$, and attained if and only if the vector $\bm{x}=(x_{i})$ has components
\begin{equation}
\begin{aligned}
x_{k}
&=
\alpha\bm{a}_{k}^{-}\bm{p},
\\
x_{j}
&\leq
\alpha a_{sj}^{-1}p_{s},
\quad
j\ne k,
\end{aligned}
\label{I-xkalphaakp}
\end{equation}
for all $\alpha>\mathbb{0}$ and indices $k$ and $s$ given by
$$
k
=
\arg\max_{1\leq i\leq n}\bm{q}^{-}\bm{b}_{i}\bm{a}_{i}^{-}\bm{p},
\qquad
s
=
\arg\max_{1\leq i\leq m}a_{ik}^{-1}p_{i}.
$$
\end{theorem}
\begin{proof}
To verify the statement, we first show that \eqref{E-DeltaqBAp} is an upper bound for the objective function in problem \eqref{P-maxqBxAxp}. Then we validate that the regular vectors $\bm{x}$ defined as \eqref{I-xkalphaakp} yield the bound, whereas any other vector does not.

Obviously, if a vector $\bm{x}$ is a solution to \eqref{P-maxqBxAxp}, then any vector $\alpha\bm{x}$ for all $\alpha>\mathbb{0}$ is also a solution, and hence the solution to the problem is scale-invariant.

Since it holds $\bm{x}(\bm{A}\bm{x})^{-}=(\bm{A}\bm{x}\bm{x}^{-})^{-}\leq\bm{A}^{-}$ provided that both $\bm{x}$ and $\bm{A}$ have no zero elements, we immediately obtain
$$
\bm{q}^{-}\bm{B}\bm{x}(\bm{A}\bm{x})^{-}\bm{p}
\leq
\bm{q}^{-}\bm{B}\bm{A}^{-}\bm{p}
=
\Delta.
$$

To find vectors that give the bound, we have to solve an equation 
$$
\bm{q}^{-}\bm{B}\bm{x}(\bm{A}\bm{x})^{-}\bm{p}
=
\Delta.
$$

With an auxiliary variable $\alpha>\mathbb{0}$, the equation is immediately transformed into a system of equations
\begin{align*}
\bm{q}^{-}\bm{B}\bm{x}
&=
\alpha\Delta,
\\
(\bm{A}\bm{x})^{-}\bm{p}
&=
\alpha^{-1}.
\end{align*}

Considering that the solution is scale-invariant, we eliminate $\alpha$ to get
\begin{equation}
\begin{aligned}
\bm{q}^{-}\bm{B}\bm{x}
&=
\Delta,
\\
(\bm{A}\bm{x})^{-}\bm{p}
&=
\mathbb{1}.
\end{aligned}
\label{E-qBxDeltaAxp1}
\end{equation}

Furthermore, we examine all solutions of the first equation at \eqref{E-qBxDeltaAxp1} to find those solutions that satisfy the second equation as well.

Due to Lemma~\ref{L-axd}, the solution of the first equation in the system is actually a family of solutions defined for each $i=1,\ldots,n$ as vectors with components
\begin{align*}
x_{i}
&=
(\bm{q}^{-}\bm{b}_{i})^{-1}\Delta,
\\
x_{j}
&\leq
(\bm{q}^{-}\bm{b}_{j})^{-1}\Delta,
\quad
j\ne i.
\end{align*}

We consider the upper bound $\Delta$ and put it into the form
$$
\Delta
=
\bm{q}^{-}\bm{B}\bm{A}^{-}\bm{p}
=
\bigoplus_{i=1}^{n}\bm{q}^{-}\bm{b}_{i}\bm{a}_{i}^{-}\bm{p}
=
\bm{q}^{-}\bm{b}_{k}\bm{a}_{k}^{-}\bm{p},
$$
where $k$ is the index of a maximum term $\bm{q}^{-}\bm{b}_{i}\bm{a}_{i}^{-}\bm{p}$ over all $i=1,\ldots,n$.

As the starting point to get a common solution to both equations \eqref{E-qBxDeltaAxp1}, we use the solution of the first equation for $i=k$, which is given by
\begin{align*}
x_{k}
&=
(\bm{q}^{-}\bm{b}_{k})^{-1}\Delta
=
\bm{a}_{k}^{-}\bm{p},
\\
x_{j}
&\leq
(\bm{q}^{-}\bm{b}_{j})^{-1}\Delta
=
(\bm{q}^{-}\bm{b}_{j})^{-1}\bm{q}^{-}\bm{b}_{k}\bm{a}_{k}^{-}\bm{p},
\quad
j\ne k.
\end{align*}

Now we examine the left hand side of the second equation at \eqref{E-qBxDeltaAxp1}. We express the vector $\bm{A}\bm{x}$ as a linear combination of columns in the matrix $\bm{A}$,
$$
\bm{A}\bm{x}
=
x_{1}\bm{a}_{1}\oplus\cdots\oplus x_{n}\bm{a}_{n}.
$$

Then we take $x_{k}=\bm{a}_{k}^{-}\bm{p}$, and consider the term $x_{k}\bm{a}_{k}=\bm{a}_{k}\bm{a}_{k}^{-}\bm{p}$. First we write
$$
\bm{a}_{k}^{-}\bm{p}
=
a_{1k}^{-1}p_{1}\oplus\cdots\oplus a_{lk}^{-1}p_{l}
=
a_{sk}^{-1}p_{s},
$$
where $s$ is the index of a maximum term $a_{ik}^{-1}p_{i}$ over all $i=1,\ldots,m$.

Since the vector $x_{k}\bm{a}_{k}=\bm{a}_{k}\bm{a}_{k}^{-}\bm{p}$ has components
\begin{align*}
x_{k}a_{sk}
&=
a_{sk}a_{sk}^{-1}p_{s}
=
p_{s},
\\
x_{k}a_{jk}
&=
a_{jk}a_{sk}^{-1}p_{s}
\geq
p_{j},
\quad
j\ne s,
\end{align*}
we arrive at a vector inequality $\bm{A}\bm{x}\geq x_{k}\bm{a}_{k}\geq\bm{p}$.

To satisfy the second equation at \eqref{E-qBxDeltaAxp1}, the vector inequality must hold as an equality for at least one component.

By taking $x_{j}$ for all $j\ne k$ to meet the condition $x_{j}\leq a_{sj}^{-1}p_{s}$, we get
\begin{align*}
a_{s1}x_{1}\oplus\cdots\oplus a_{sn}x_{n}
&=
p_{s},
\\
a_{i1}x_{1}\oplus\cdots\oplus a_{in}x_{n}
&\geq
p_{i},
\quad
i\ne s.
\end{align*}

With the inequality $a_{sj}^{-1}p_{s}\leq\bm{a}_{j}^{-}\bm{p}\leq(\bm{q}^{-}\bm{b}_{j})^{-1}\bm{q}^{-}\bm{b}_{k}\bm{a}_{k}^{-}\bm{p}$, we conclude that any vector $\bm{x}$ with components
\begin{align*}
x_{k}
&=
\bm{a}_{k}^{-}\bm{p},
\\
x_{j}
&\leq
a_{sj}^{-1}p_{s},
\quad
j\ne k,
\end{align*}
presents a common solution of both equations at \eqref{E-qBxDeltaAxp1}, and so a solution to \eqref{P-maxqBxAxp}. Taking into account that the solution is scale-invariant, we get \eqref{I-xkalphaakp}.

Finally, we show that the solutions to the first equation for each $i\ne k$ cannot satisfy the second equation. We assume that $\bm{q}^{-}\bm{b}_{i}\bm{a}_{i}^{-}\bm{p}<\bm{q}^{-}\bm{b}_{k}\bm{a}_{k}^{-}\bm{p}$, and consider the solution
\begin{align*}
x_{i}
&=
(\bm{q}^{-}\bm{b}_{i})^{-1}\bm{q}^{-}\bm{b}_{k}\bm{a}_{k}^{-}\bm{p},
\\
x_{j}
&\leq
(\bm{q}^{-}\bm{b}_{j})^{-1}\bm{q}^{-}\bm{b}_{k}\bm{a}_{k}^{-}\bm{p},
\quad
j\ne i.
\end{align*}

With the assumption, we have $(\bm{q}^{-}\bm{b}_{i})^{-1}\bm{q}^{-}\bm{b}_{k}\bm{a}_{k}^{-}\bm{p}>\bm{a}_{i}^{-}\bm{p}$, and therefore,
$$
x_{i}\bm{a}_{i}
=
\bm{a}_{i}(\bm{q}^{-}\bm{b}_{i})^{-1}\bm{q}^{-}\bm{b}_{k}\bm{a}_{k}^{-}\bm{p}
>
\bm{a}_{i}\bm{a}_{i}^{-}\bm{p}.
$$

Now we write $\bm{A}\bm{x}\geq x_{i}\bm{a}_{i}>\bm{a}_{i}\bm{a}_{i}^{-}\bm{p}$, which yields $(\bm{A}\bm{x})^{-}<(\bm{a}_{i}^{-}\bm{p})^{-1}\bm{a}_{i}^{-}$. Finally, we see that $(\bm{A}\bm{x})^{-}\bm{p}<(\bm{a}_{i}^{-}\bm{p})^{-1}\bm{a}_{i}^{-}\bm{p}=\mathbb{1}$, and thus the above solution fails to solve the entire problem.
\qed
\end{proof}

\subsection{Particular cases}

Now we present a particular case of the general problem, which involves idempotent analogues of the vector and matrix norms. Another particular case is considered in the next section in the context of solution of scheduling problems.

Let us assume that $\bm{p}=\bm{q}=\mathbb{1}$ and note that $\mathbb{1}^{T}\bm{a}_{i}=\|\bm{a}_{i}\|$ and $\bm{b}_{i}^{-}\mathbb{1}=\|\bm{b}_{i}^{-}\|$. Moreover, we have
$$
\mathbb{1}^{T}\bm{B}\bm{x}(\bm{A}\bm{x})^{-}\mathbb{1}
=
\|\bm{B}\bm{x}\|\|(\bm{A}\bm{x})^{-}\|,
\qquad
\mathbb{1}^{T}\bm{B}\bm{A}^{-}\mathbb{1}
=
\|\bm{B}\bm{A}^{-}\|.
$$

Under these assumptions, problem \eqref{P-maxqBxAxp} takes the form
\begin{equation}
\begin{aligned}
&
\text{maximize}
&&
\|\bm{B}\bm{x}\|\|(\bm{A}\bm{x})^{-}\|.
\end{aligned}
\label{P-max1BxAx1}
\end{equation}

It follows from Theorem~\ref{T-maxqBxAxp} that a solution to problem \eqref{P-max1BxAx1} goes as follows.

\begin{corollary}\label{C-max1BxAx1}
Suppose $\bm{A}$ is a matrix with regular columns and $\bm{B}$ is a column regular matrix. Define a scalar
\begin{equation*}
\Delta
=
\|\bm{B}\bm{A}^{-}\|.
\label{E-DeltaBA}
\end{equation*}

Then the maximum in problem \eqref{P-max1BxAx1} is equal to $\Delta$, and attained if and only if the vector $\bm{x}=(x_{i})$ has components
\begin{equation*}
\begin{aligned}
x_{k}
&=
\alpha\|\bm{a}_{k}^{-}\|,
\\
x_{j}
&\leq
\alpha a_{sj}^{-1},
\quad
j\ne k,
\end{aligned}
\label{I-xkalphabk}
\end{equation*}
for all $\alpha>\mathbb{0}$ and indices $k$ and $s$ given by
$$
k
=
\arg\max_{1\leq i\leq n}\|\bm{b}_{i}\|\|\bm{a}_{i}^{-}\|,
\qquad
s
=
\arg\max_{1\leq i\leq m}a_{ik}^{-1}.
$$
\end{corollary}

\subsection{Extension to constrained problems}

The solution to problem \eqref{P-maxqBxAxp} can be well extended to cover certain constrained problems. Specifically, assume $\bm{C}\in\mathbb{X}^{n\times n}$ to be given and consider a problem
\begin{equation}
\begin{aligned}
&
\text{maximize}
&&
\bm{q}^{-}\bm{B}\bm{x}(\bm{A}\bm{x})^{-}\bm{p},
\\
&
\text{subject to}
&&
\bm{C}\bm{x}
\leq
\bm{x}.
\end{aligned}
\label{P-maxqBxAxpCxlex}
\end{equation}

By Theorem~\ref{T-I-Axx}, the inequality constraint in \eqref{P-maxqBxAxpCxlex} has regular solutions only when $\mathop\mathrm{Tr}(\bm{C})\leq\mathbb{1}$. Under this condition, the solution is given  by $\bm{x}=\bm{C}^{\ast}\bm{u}$ for all regular vectors $\bm{u}$, whereas the entire problem reduces to
\begin{equation*}
\begin{aligned}
&
\text{maximize}
&&
\bm{q}^{-}\bm{B}\bm{C}^{\ast}\bm{u}(\bm{A}\bm{C}^{\ast}\bm{u})^{-}\bm{p}.
\end{aligned}
\label{P-maxqBCuACup-ast}
\end{equation*}

The unconstrained problem admits an immediate solution based on Theorem~\ref{T-maxqBxAxp}, provided that the matrix $\bm{A}\bm{C}^{\ast}$ has only regular columns and the matrix $\bm{B}\bm{C}^{\ast}$ is column regular.

Since it holds that $\bm{C}^{\ast}\geq\bm{I}$, the condition is fulfilled when the matrix $\bm{A}$ has no zero entries and $\bm{B}$ is column regular. The assumption on $\bm{A}$, however, is not necessary to apply the theorem. Specifically, the condition is also satisfied if the matrix $\bm{A}$ is row regular, whereas $\bm{C}$ is irreducible. Indeed, in this case, the matrix $\bm{C}^{\ast}$ and, thus the matrix $\bm{A}\bm{C}^{\ast}$, have no zero entries.

It is clear that the condition for $\bm{A}$ to be row regular is necessary.

Note that the solution to the unconstrained problem is given by Theorem~\ref{T-maxqBxAxp} in terms of the auxiliary vector $\bm{u}$ and, therefore, needs to be translated into a solution with respect to $\bm{x}$ with the mapping $\bm{x}=\bm{C}^{\ast}\bm{u}$.

Examples of solutions to particular constrained problems drawn from project scheduling are given in the next section.

\section{Application to project scheduling}\label{S-Aps}

In this section, we revisit scheduling problems \eqref{P-maxyimaxyixjaijyi}, \eqref{P-maxximaxxixjcijxi}, and \eqref{P-maxyimaxyixjcijxixjaijyi} to reformulate and solve them as optimization problems in the tropical mathematics setting. To illustrate the results obtained, numerical examples are also given.

\subsection{Representation and solution of problems}

Taking into account that the representation of the problems in the ordinary notation involves maximum, addition, and additive inversion, we translate them into the language of the semifield $\mathbb{R}_{\max,+}$.

We start with problem \eqref{P-maxyimaxyixjaijyi}, which can be written in terms of $\mathbb{R}_{\max,+}$ in a scalar form as
\begin{equation*}
\begin{aligned}
&
\text{maximize}
&&
\left(\bigoplus_{i=1}^{n}y_{i}\right)
\left(\bigoplus_{i=1}^{n}y_{i}^{-1}\right),
\\
&
\text{subject to}
&&
\bigoplus_{j=1}^{n}a_{ij}x_{j}
=
y_{i},
\quad
i=1,\ldots,n.
\end{aligned}
\end{equation*}

We introduce a matrix $\bm{A}=(a_{ij})$ and vectors $\bm{x}=(x_{i})$ and $\bm{y}=(y_{i})$ to shift from the scalar representation to that in the matrix-vector notation
\begin{equation}
\begin{aligned}
&
\text{maximize}
&&
\|\bm{y}\|\|\bm{y}^{-}\|,
\\
&
\text{subject to}
&&
\bm{A}\bm{x}
=
\bm{y}.
\end{aligned}
\label{P-maxyyAxy}
\end{equation}

A complete solution to the problem is given as follows.

\begin{lemma}\label{L-maxyyAxy}
Suppose $\bm{A}$ is a matrix with regular columns. Define a scalar
\begin{equation*}
\Delta
=
\|\bm{A}\bm{A}^{-}\|.
\label{E-DeltaAA}
\end{equation*}

Then the maximum in problem \eqref{P-maxyyAxy} is equal to $\Delta$, and attained if and only if the vector $\bm{x}=(x_{i})$ has components
\begin{equation*}
\begin{aligned}
x_{k}
&=
\alpha\|\bm{a}_{k}^{-}\|,
\\
x_{j}
&\leq
\alpha a_{sj}^{-1},
\quad
j\ne k,
\end{aligned}
\label{I-xkalphaak}
\end{equation*}
for all $\alpha>\mathbb{0}$ and indices $k$ and $s$ given by
$$
k
=
\arg\max_{1\leq i\leq n}\|\bm{a}_{i}\|\|\bm{a}_{i}^{-}\|,
\qquad
s
=
\arg\max_{1\leq i\leq n}a_{ik}^{-1}.
$$
\end{lemma}
\begin{proof}
By substitution $\bm{y}=\bm{A}\bm{x}$, we get an unconstrained problem in the form of \eqref{P-max1BxAx1}. Application of Corollary~\ref{C-max1BxAx1} with $\bm{B}=\bm{A}$ completes the solution.
\qed
\end{proof}

Note that the solution is actually determined up to a nonzero factor, and so can serve as a basis for further optimization of a schedule under additional constraints, including due date and early start time constraints.

Furthermore, we examine problem \eqref{P-maxximaxxixjcijxi}. When expressed in terms of the operations in the semifield $\mathbb{R}_{\max,+}$, the problem becomes
\begin{equation*}
\begin{aligned}
&
\text{maximize}
&&
\left(\bigoplus_{i=1}^{n}x_{i}\right)
\left(\bigoplus_{i=1}^{n}x_{i}^{-1}\right),
\\
&
\text{subject to}
&&
\bigoplus_{j=1}^{n}c_{ij}x_{j}
\leq
x_{i},
\quad
i=1,\ldots,n.
\end{aligned}
\end{equation*}

With a matrix $\bm{C}=(c_{ij})$, we switch to matrix-vector notation and get
\begin{equation}
\begin{aligned}
&
\text{maximize}
&&
\|\bm{x}\|\|\bm{x}^{-}\|,
\\
&
\text{subject to}
&&
\bm{C}\bm{x}
\leq
\bm{x}.
\end{aligned}
\label{P-maxxxCxx}
\end{equation}

\begin{lemma}\label{L-maxxxCxx}
Suppose $\bm{C}$ is an irreducible matrix with $\mathop\mathrm{Tr}(\bm{C})\leq\mathbb{1}$. Define a scalar
\begin{equation*}
\Delta
=
\|\bm{C}^{\ast}(\bm{C}^{\ast})^{-}\|.
\end{equation*}

Then the maximum in problem \eqref{P-maxxxCxx} is equal to $\Delta$, and attained if and only if $\bm{x}=\bm{C}^{\ast}\bm{u}$, where $\bm{u}=(u_{i})$ is any vector with components
\begin{equation*}
\begin{aligned}
u_{k}
&=
\alpha\|(\bm{c}_{k}^{\ast})^{-}\|,
\\
u_{j}
&\leq
\alpha (c_{sj}^{\ast})^{-1},
\quad
j\ne k,
\end{aligned}
\label{I-ukalphack}
\end{equation*}
for all $\alpha>\mathbb{0}$ and indices $k$ and $s$ given by
$$
k
=
\arg\max_{1\leq i\leq n}\|\bm{c}_{i}^{\ast}\|\|(\bm{c}_{i}^{\ast})^{-}\|,
\qquad
s
=
\arg\max_{1\leq i\leq n}(c_{ik}^{\ast})^{-1}.
$$
\end{lemma}
\begin{proof}
It follows from Theorem~\ref{T-I-Axx} that each solution to the inequality constraint in \eqref{P-maxxxCxx} is given by $\bm{x}=\bm{C}^{\ast}\bm{u}$, where $\bm{u}$ is a regular vector. Taking the general solution instead of the inequality, we arrive at an optimization problem with respect to $\bm{u}$ in the form of \eqref{P-maxyyAxy} with $\bm{A}=\bm{C}^{\ast}$. After solution of the last problem according to Lemma \ref{L-maxyyAxy}, we arrive at the desired result.
\qed 
\end{proof}

Finally, in a similar way as above, problem \eqref{P-maxyimaxyixjcijxixjaijyi} is represented in the form
\begin{equation}
\begin{aligned}
&
\text{maximize}
&&
\|\bm{y}\|\|\bm{y}^{-}\|,
\\
&
\text{subject to}
&&
\bm{A}\bm{x}
=
\bm{y},
\\
&
&&
\bm{C}\bm{x}
\leq
\bm{x},
\end{aligned}
\label{P-maxyyCxxAxy}
\end{equation}
and then accepts a complete solution given by the next result.

\begin{lemma}\label{L-maxyyCxxAxy}
Suppose $\bm{A}$ is a row regular matrix and $\bm{C}$ a matrix with $\mathop\mathrm{Tr}(\bm{C})\leq\mathbb{1}$ such that all columns in the matrix $\bm{D}=\bm{A}\bm{C}^{\ast}$ are regular. Define a scalar
\begin{equation*}
\Delta
=
\|\bm{D}\bm{D}^{-}\|.
\end{equation*}

Then the maximum in problem \eqref{P-maxyyCxxAxy} is equal to $\Delta$, and attained if and only if $\bm{x}=\bm{C}^{\ast}\bm{u}$, where $\bm{u}=(u_{i})$ is any vector with components
\begin{equation*}
\begin{aligned}
u_{k}
&=
\alpha\|\bm{d}_{k}^{-}\|,
\\
u_{j}
&\leq
\alpha d_{sj}^{-1},
\quad
j\ne k,
\end{aligned}
\label{I-ukalphadk}
\end{equation*}
for all $\alpha>\mathbb{0}$ and indices $k$ and $s$ given by
$$
k
=
\arg\max_{1\leq i\leq n}\|\bm{d}_{i}\|\|\bm{d}_{i}^{-}\|,
\qquad
s
=
\arg\max_{1\leq i\leq n}d_{ik}^{-1}.
$$
\end{lemma}

Note that the matrix $\bm{D}=\bm{A}\bm{C}^{\ast}$ has only regular columns when all columns in the matrix $\bm{A}$ are regular or the matrix $\bm{C}$ is irreducible.

\subsection{Numerical examples}

We start with problem \eqref{P-maxyyAxy}, which is to maximize the deviation of completion time. Consider a project with $n=3$ activities operating under start-finish constraints given by a matrix 
$$
\bm{A}
=
\left(
\begin{array}{ccc}
4 & 1 & 1
\\
2 & 2 & 0
\\
0 & 1 & 3
\end{array}
\right).
$$

To apply Lemma~\ref{L-maxyyAxy}, we calculate
$$
\bm{A}^{-}
=
\left(
\begin{array}{rrr}
-4 & -2 & 0
\\
-1 & -2 & -1
\\
-1 & 0 & -3
\end{array}
\right),
\quad
\bm{A}\bm{A}^{-}
=
\left(
\begin{array}{ccc}
0 & 2 & 4
\\
1 & 0 & 2
\\
2 & 3 & 0
\end{array}
\right),
\quad
\Delta
=
\|\bm{A}\bm{A}^{-}\|
=
4.
$$

Furthermore, we get
$$
\|\bm{a}_{1}\|\|\bm{a}_{1}^{-}\|
=
4,
\qquad
\|\bm{a}_{2}\|\|\bm{a}_{2}^{-}\|
=
1,
\qquad
\|\bm{a}_{3}\|\|\bm{a}_{3}^{-}\|
=
3,
$$
and then verify that
$$
\|\bm{a}_{1}\|\|\bm{a}_{1}^{-}\|
=
\max\{\|\bm{a}_{i}\|\|\bm{a}_{i}^{-}\||\ i=1,2,3\},
\quad
a_{31}^{-1}
=
\max\{a_{i1}^{-1}|\ i=1,2,3\}.
$$

Taking $k=1$ and $s=3$, we assume $\alpha=0$ to obtain a solution set defined by the relations
$$
x_{1}
=
0,
\qquad
x_{2}
\leq
-1,
\qquad
x_{3}
\leq
-3.
$$

Specifically, the solution vector with the latest initiation time and its related vector of completion time are given by
$$
\bm{x}
=
\left(
\begin{array}{r}
0
\\
-1
\\
-3
\end{array}
\right),
\qquad
\bm{y}
=
\bm{A}\bm{x}
=
\left(
\begin{array}{c}
4
\\
2
\\
0
\end{array}
\right).
$$

To illustrate the solution to problem \eqref{P-maxxxCxx} given by Lemma~\ref{L-maxxxCxx}, we examine a project with start-start precedence constraints defined by a matrix
$$
\bm{C}
=
\left(
\begin{array}{rrr}
\mathbb{0} & -2 & 1
\\
0 & \mathbb{0} & 2
\\
-1 & \mathbb{0} & \mathbb{0}
\end{array}
\right),
$$
where the symbol $\mathbb{0}=-\infty$ is used to save space.

First we successively find
$$
\bm{C}^{2}
=
\left(
\begin{array}{rrr}
0 & \mathbb{0} & 0
\\
1 & -2 & 1
\\
\mathbb{0} & -3 & 0
\end{array}
\right),
\qquad
\bm{C}^{3}
=
\left(
\begin{array}{rrr}
-1 & -2 & 1
\\
0 & -1 & 2
\\
-1 & \mathbb{0} & -1
\end{array}
\right),
\qquad
\mathop\mathrm{Tr}(\bm{C})
=
0,
$$
and then form the matrices
$$
\bm{C}^{\ast}
=
\bm{I}\oplus\bm{C}\oplus\bm{C}^{2}
=
\left(
\begin{array}{rrr}
0 & -2 & 1
\\
1 & 0 & 2
\\
-1 & -3 & 0
\end{array}
\right),
\qquad
(\bm{C}^{\ast})^{-}
=
\left(
\begin{array}{rrr}
0 & -1 & 1
\\
2 & 0 & 3
\\
-1 & -2 & 0
\end{array}
\right).
$$

Furthermore, we calculate
$$
\bm{C}^{\ast}(\bm{C}^{\ast})^{-}
=
\left(
\begin{array}{rrr}
0 & -1 & 1
\\
2 & 0 & 3
\\
-1 & -2 & 0
\end{array}
\right),
\qquad
\Delta
=
\|\bm{C}^{\ast}(\bm{C}^{\ast})^{-}\|
=
3.
$$

We examine columns in the matrix $\bm{C}^{\ast}$ to get
$$
\|\bm{c}_{1}^{\ast}\|\|(\bm{c}_{1}^{\ast})^{-}\|
=
2,
\qquad
\|\bm{c}_{2}^{\ast}\|\|(\bm{c}_{2}^{\ast})^{-}\|
=
3,
\qquad
\|\bm{c}_{3}^{\ast}\|\|(\bm{c}_{3}^{\ast})^{-}\|
=
2.
$$

We take $k=2$ and then identify $s=3$. With $\alpha=0$, we arrive at a set of solutions $\bm{x}=\bm{C}^{\ast}\bm{u}$, where $\bm{u}=(u_{i})$ is a vector with components
$$
u_{1}
\leq
1,
\qquad
u_{2}
=
3,
\qquad
u_{3}
\leq
0.
$$

For the solution with the latest initiation time, we have
$$
\bm{u}
=
\left(
\begin{array}{c}
1
\\
3
\\
0
\end{array}
\right),
\qquad
\bm{x}
=
\bm{C}^{\ast}\bm{u}
=
\left(
\begin{array}{c}
1
\\
3
\\
0
\end{array}
\right).
$$

Now we apply Lemma~\ref{L-maxyyCxxAxy} to solve problem \eqref{P-maxyyCxxAxy}, which is to maximize deviation between completion times of activities in a project with a combined set of precedence constraints. We consider a project with $n=3$ activities, where start-finish and start-start constraints are given by respective matrices
$$
\bm{A}
=
\left(
\begin{array}{ccc}
4 & 1 & 1
\\
2 & 2 & 0
\\
0 & 1 & 3
\end{array}
\right),
\qquad
\bm{C}
=
\left(
\begin{array}{rrr}
\mathbb{0} & -2 & 1
\\
0 & \mathbb{0} & 2
\\
-1 & \mathbb{0} & \mathbb{0}
\end{array}
\right).
$$

Using the result of the previous example, we find the matrix
$$
\bm{D}
=
\bm{A}\bm{C}^{\ast}
=
\left(
\begin{array}{ccc}
4 & 1 & 1
\\
2 & 2 & 0
\\
0 & 1 & 3
\end{array}
\right)
\left(
\begin{array}{rrr}
0 & -2 & 1
\\
1 & 0 & 2
\\
-1 & -3 & 0
\end{array}
\right)
=
\left(
\begin{array}{rrr}
4 & 2 & 5
\\
3 & 2 & 4
\\
2 & 1 & 3
\end{array}
\right).
$$ 

Furthermore, we obtain
$$
\bm{D}^{-}
=
\left(
\begin{array}{rrr}
-4 & -3 & -2
\\
-2 & -2 & -1
\\
-5 & -4 & -3
\end{array}
\right),
\quad
\bm{D}\bm{D}^{-}
=
\left(
\begin{array}{rrr}
0 & 1 & 2
\\
0 & 0 & 1
\\
-1 & -1 & 0
\end{array}
\right),
\quad
\Delta
=
\|\bm{D}\bm{D}^{-}\|
=
2.
$$

Analysis of columns in the matrix $\bm{D}$ gives
$$
\|\bm{d}_{1}\|\|\bm{d}_{1}^{-}\|
=
2,
\qquad
\|\bm{d}_{2}\|\|\bm{d}_{2}^{-}\|
=
1,
\qquad
\|\bm{d}_{3}\|\|\bm{d}_{3}^{-}\|
=
2.
$$

First we take $k=1$ and $s=3$. With $\alpha=0$, we get a solution $\bm{x}=\bm{C}^{\ast}\bm{u}$, where the vector $\bm{u}=(u_{i})$ has components
$$
u_{1}
=
-2,
\qquad
u_{2}
\leq
-1,
\qquad
u_{3}
\leq
-3.
$$

The solution with the latest initiation times is given by
$$
\bm{u}
=
\left(
\begin{array}{c}
-2
\\
-1
\\
-3
\end{array}
\right),
\qquad
\bm{x}
=
\bm{C}^{\ast}\bm{u}
=
\left(
\begin{array}{c}
-2
\\
-1
\\
-3
\end{array}
\right),
\qquad
\bm{y}
=
\bm{D}\bm{u}
=
\left(
\begin{array}{c}
2
\\
1
\\
0
\end{array}
\right).
$$

Another solution is obtained by setting $k=3$ and $s=3$. The vector $\bm{u}$ is then defined by
$$
u_{1}
\leq
-2,
\qquad
u_{2}
\leq
-1,
\qquad
u_{3}
=
-3.
$$

The solution with latest initiation time is obviously the same as before.

\bibliographystyle{utphys}

\bibliography{A_maximization_problem_in_tropical_mathematics_a_complete_solution_and_application_examples}

\end{document}